\documentclass[numbook]{siamltex}
\usepackage[T1]{fontenc}
\usepackage[intlimits]{amsmath}
\usepackage{amssymb}
\usepackage{amsfonts}
\usepackage[latin1,utf8x]{inputenc}
\usepackage{geometry}
\usepackage{graphics}
\usepackage{natbib}
\usepackage{multirow,ctable,colortbl}
\usepackage{algorithm}

\newcommand{\Om}{\Omega}
\newcommand{\RR}{\mathbb{R}}
\newcommand{\Rn}{{\mathbb{R}^N}}

\newcommand{\Div}{\hbox{\rm div}}
\newcommand{\eps}{\varepsilon}

\newcommand{\LN}{{\cal L}_N}

\newcommand{\gr}{\nabla}

\newcommand{\pto}[2]{\langle #1 , #2\rangle}

\newcommand{\weakstar}{{\rightharpoonup}^{\hspace{-7pt}*}\ }

\renewcommand\theenumi{(\roman{enumi})}
\renewcommand\theenumii{(\alph{enumii})}

\begin{document}

\title{Existence and Approximability Results for variational problems under uniform constraints on the gradient
by power penalty.
\thanks{This work was partially supported by the Millennium Science Initiative of MIDEPLAN-Chile, and by CONICYT-Chile
under grants FONDAP in Applied Mathematics and FONDECYT 1050706.}
}

\author{Felipe Alvarez\thanks{Centro de Modelamiento Matematico (CNRS UMI 2807), Departamento de Ingenier\'ia Matematica, 
Universidad de Chile, Av. Blanco Encalada 2120, Santiago, Chile.}        \and
        Salvador Flores\thanks{Centro de Modelamiento Matem\'atico (CNRS UMI 2807) -- Universidad de Chile. 
	Supported by Fondecyt under grant 3120166}
}

\maketitle

\begin{abstract}
 Variational problems under uniform quasiconvex constraints on the gradient are studied. 
In particular, existence of solutions to such problems is proved as well as existence of lagrange multipliers associated 
to the uniform constraint. They are shown to satisfy an Euler-Lagrange equation and a complementarity property.
Our technique consists in approximating the original problem by a one-parameter family of smooth unconstrained optimization problems.
Numerical experiments confirm the ability of our method to accurately compute solutions and Lagrange multipliers.
\end{abstract}
\begin{AMS}
{49M30,49J45.} 
\end{AMS}

\section{Introduction}\label{intro}                                  %

We study the following class of problems from the calculus of variations
\begin{equation}\label{P:main}
\inf\{J(v): |T(x,\nabla v(x))|\leq 1\ a.e\ x\ in\ \Om, ~v=g\ on\ \partial \Om\}.
\end{equation}
In particular, we prove existence and approximability of solutions and Lagrange multipliers associated to the uniform constraint on
the gradient. We approximate the problem by a sequence of unconstrained problems penalizing the uniform constraint by a power term.

The model case of  \eqref{P:main} is the problem of the elastoplastic torsion of a cilindrical bar of section $\Omega$:
\begin{equation}\label{P:torsion}
 \min_{v \in K_0}   \frac{1}{2} \int_\Omega (|\nabla v(x)|^2 - h(x)v(x) )dx 
\end{equation}
\noindent for $K_0=\{v\in H_0^1(\Omega) \mid |\nabla v(x) |\leq 1\ a.e\ x\in \Om \}$. 
Problem \eqref{P:torsion} 
has been extensively studied by \citet{Ting,brezis,cafarelli} and in the numerical aspects by \citet{GLT}. 
\citet{brezis} proves the existence and uniqueness of a multiplier $\lambda\in L^\infty$ satisfying
the system
\begin{subequations}\label{E:bresystem}
\begin{eqnarray}
 \lambda \geq 0\qquad a.e\ on\ \Omega\\
\lambda (1-|\nabla u|)=0 \qquad a.e\ on\ \Omega \label{bresystem:complementarity}\\
-\Delta u-\sum_{i=1}^N \frac{\partial }{\partial x_i} (\lambda \frac{\partial u}{\partial x_i})=h \quad in\ \mathcal{D}' \label{bresystem:euler-lagrange}
\end{eqnarray}
\end{subequations}

when the right hand side $h$ is constant. \citet{percivale} reconsider the problem for a general elliptic operator $\mathcal{A}$
and nonconstant right hand side $h$, obtaining a measure multiplier satisfying a system  analogous to 
 \eqref{bresystem:complementarity}-\eqref{bresystem:euler-lagrange}.
 \citet{brezis} uses the characteristics method to solve \eqref{bresystem:euler-lagrange} for $\lambda$, obtaining a semi-explicit formula 
for the multiplier. \citet{percivale} approximate the problem by a sequence of nonsmooth problems penalizing the violation of the 
constraint $|\nabla u|\leq 1\ a.e$.
Whether similar results could be obtained in the framework of a general duality theory standed as an open question 
for a long time. \citet{Ekeland} show the insufficiency of the traditional duality theory for tackling this problem.
The question was solved positively by \citet{daniele} using a new infinite dimensional duality theory \citep[see also][]{Donato2011,Maugeri2014}.
\citet{daniele} show, for a large class of problems including Problems \eqref{e.prob} and \eqref{P:torsion}, that if the problem 
is solvable and the solution satisfies a constraint qualification condition, then there exists a Lagrange 
multiplier  $\lambda\in L^\infty_+$ satisfying \eqref{bresystem:complementarity}, which is indeed the solution of a dual problem. 
Concerning existence of solutions for the general  Problem \eqref{P:main}, we can cite the results of 
\citet{Ball}, showing existence for variational problems under constraints of the type $T(\nabla v(x))\in C(x)$ for almost every $x \in \Om$ .
From this perspective, the existence of solutions as well as of Lagrange multipliers is well established. 
Nonetheless, at least two issues remain unsolved. The first is to have a practical way to approximate Problem \eqref{P:main} by 
simpler problems that ca be solved using existing mature numerical methods. The second issue is closely related to the first, and 
has to do with choosing a particular solution in problems with lack of uniqueness. In this paper we address those 
open issues by providing an approximation scheme for Problem \eqref{P:main}. The original problem is approximated by 
a sequence of unconstrained problems whose solution converges to a solution of the constrained problem. 
Moreover, by analyzing the optimality conditions we identify a term that is then showed to converge to a Lagrange multiplier 
associated to the uniform constraint on the gradient. In this way, we recover and in some cases improve the existence results
and provide a practical approximation scheme. The effectiveness of our approach is illustrated through numerical 
simulations.    

\section{Statement of the problem and main results}

Let $\Omega$ be a bounded domain in $\Rn$ with $N\geq 1$ and $T:\Omega\times\RR^{m\times N}\to[0,\infty[$ a Carath\'eodory function. 
 Let $s\geq 1$ and consider a functional $J:W^{1,s}(\Omega;\RR^m)\to \RR\cup\{+\infty\}$, which is supposed to be
bounded from below and sequentially lower semicontinuous in the weak topology of $W^{1,s}(\Omega;\RR^m)$.
We are interested in the minimization problem
\begin{equation}\label{e.prob}
\inf\{J(v)\mid \|T(\cdot,\nabla v) \|_{\infty,\Omega}\leq 1, ~v \in g+ W_0^{1,s}(\Omega;\RR^m)\},
\end{equation}
where 
$$ \|T(\cdot,\nabla v)\|_{\infty,\Omega}=\hbox{ess-}\sup\{T(x,\nabla v(x)) \mid x\in\Omega\},$$
and $g\in W^{1,\infty}(\Omega;\RR^m)\cap C(\overline{\Omega};\RR^m)$
is a given function satisfying
\begin{equation}\label{E:g}
J(g)<+\infty \hbox{ and }T(x,\nabla g(x))\leq 1\ \hbox{ for a.e.}\ x\in\overline\Omega.
\end{equation}
Define $J_\infty: W^{1,s}(\Omega;\RR^m)\to \RR\cup\{+\infty\}$ by
$$
J_\infty(v)=\left\{
\begin{array}{cl}
J(v) & \hbox{ if } \|T(\cdot,\nabla v)\|_{\infty,\Omega}\leq 1,  \\
+\infty &\hbox{ otherwise}. \end{array} \right.
$$
Then \eqref{e.prob} may be rewritten as
\begin{equation}\label{e.prob2}
\inf\left\{J_\infty(v)\mid v \in g+ W_0^{1,s}(\Omega;\RR^m)\right\}.
\end{equation}
By \eqref{E:g}, we have that  $J_\infty(g)<+\infty$.\\

From now on, we  assume that $T$ is {\em quasiconvex} in the sense
of Morrey, {\em i.e.} for almost for every $x_0\in\Omega$ and any $\xi_0\in \RR^{m\times N}$
\begin{equation}\label{e.quasiconvex}
 T(x_0,\xi_0)\leq \frac{1}{\LN(D)}\int_D T(x_0,\xi_0+\nabla\phi(x))dx,
\end{equation}
where $D$ is an arbitrary bounded domain in $\Rn$ and $\phi$ is any function in $W^{1,\infty}_0(D;\RR^m)$. Here,
$\LN$ stands for the Lebesgue measure in  $\Rn$. 
Suppose also that
\begin{equation}\label{e.growth}
\alpha_1 (1+|\xi|^r) \leq  T(x,\xi) \leq \beta_1(1+|\xi|^{r}) 
\end{equation}
where $0<\alpha_1\leq \beta_1$  and $1\leq r< \infty$. 
Concerning the functional $J$, in most interesting applications it will take the integral form
\begin{equation}\label{E:integral-form}
J(u)=\int_\Om f(x,u(x),\nabla u(x)) dx 
\end{equation}
where $f:\RR^N \times \RR^m \times \RR^{m\times N}$ is a Carathéodory integrand satisfying, for almost every $x\in \Om$,
for every $(u,\xi)\in\RR^N\times \RR^{m\times N}$,
 
\begin{subequations}
\begin{eqnarray}
\xi \mapsto f(x,u,\xi)\ \text{is quasiconvex}\\
\gamma_1(x)\leq f(x,u,\xi) \leq \beta_2 (|\xi|^s+ |u|^t )+\gamma_2(x)
\end{eqnarray}
\end{subequations}
where $\beta_2\geq 0$, $\gamma_1,\gamma_2\in L^1(\Om)$ and $1\leq t< \infty$.\\

For each $p\in\ ]\max(r,s),\infty[$ define the $p$-power penalty functional $J_p:W^{1,p}(\Omega;\RR^m)\to\RR\cup\{+\infty\}$ by
$$
J_p(v)= J(v)+\frac{r}{p} \int_\Omega Tv(x)^{p/r} dx, 
$$
where
$$
Tv(x)=T(x,\nabla v(x))
$$
 and consider the penalized problems
\begin{equation}\label{e.prob.p}
\inf\{J_p(v)\mid v \in g+ W_0^{1,p}(\Omega;\RR^m)\}.
\end{equation}

Under the above conditions, the existence of solutions $u_p$ to \eqref{e.prob.p} follows from a
standard application of the direct method of the calculus of variations \citep[cf.][Theorem 8.29]{dacorogna}. 
In this direction, notice that the quasiconvexity of $T$ yields the quasiconvexity of $T^p$ for every $1<p<\infty$.

Any selection of solutions to Problems \eqref{e.prob.p} uniformly converges to a solution of Problem \eqref{P:main}.
We do not assume a priori existence of solutions to Problem \eqref{P:main}, therefore the following is an 
existence and approximability result.

\begin{theorem}\label{T:primal-intro}
Under the previous assumptions, we have that:
\begin{itemize}
\item[{\rm (i)}] For every $q\geq \max\{N+1,r,s\}$, the net $\{u_p\mid p\geq q,\, p\to\infty\}$ is bounded in $W^{1,q}(\Omega;\RR^m)$ and
relatively compact in $C^\alpha(\overline{\Omega};\RR^m)$, where $\alpha=1-N/q$.
\item[{\rm (ii)}] If $u_\infty$ is a cluster point of $\{u_p\mid p\to\infty\}$ in $C(\overline{\Omega};\RR^m)$, then $u_\infty$ is an optimal 
solution to \eqref{e.prob} and, moreover,
\[
\lim\limits_{p\to\infty}\min
J_p=\lim\limits_{p\to\infty}J_p(u_p)=\lim\limits_{p\to\infty}J(u_p)=J(u_\infty)=\min
J_\infty.\]
\end{itemize}
\end{theorem}

Next we address the existence and approximability of Lagrange multipliers for the uniform constraint on the gradient.
The underlying rationale bears some resemblances to some methods for showing existence of Lagrange multipliers 
without recourse to separation theorems, such as the Fritz John optimality conditions in nonlinear programming.
Let us consider the Lagrange functional $L:H^1\times L^{\infty}_+\to \overline{\RR}$
\begin{equation}\label{E:lagrangeano}
L(u,\lambda)=J(u)+\int \lambda (x)(Tu(x)-1) dx
\end{equation}

If a solution $u$ to Problem \eqref{e.prob} satisfies a constraint qualification condition, 
then there exists $\lambda\in L^\infty_+$ such that $(u,\lambda)$ is a saddle point of $L$ \citep{daniele}.
Let $(u,\lambda)$ be a saddle point of $L$, and suppose that $T$ is differentiable with respect to its second argument. 
The minimality condition for $u$ reads 
$$
J'(u)[v]+\int \lambda DT(x,\nabla u(x))\nabla v(x)=0\quad \forall v\in C_0^\infty  
$$
On the other hand, the optimality conditions  for the penalized problem \eqref{e.prob.p} yields
$$
J'(u_p)[v]+\int  (Tu_p)^{p-1} DT(x,\nabla u_p(x))\nabla v(x)=0\quad \forall v\in C_0^\infty.
$$
Suppose that $J'(u_p) \to J'(u)$ as $p\to\infty$, then

\begin{equation}\label{E:esquema}
\int  (Tu_p)^{p-1} DT(x,\nabla u_p(x))\nabla v(x)\to \int \lambda DT(x,\nabla u(x))\nabla v(x)\ \forall v\in C_0^\infty 
\end{equation}
Equation \eqref{E:esquema} strongly suggests that the sequence $\{(Tu_p)^{p-1}\}_{p\geq p_1}$ must play the role of
a Lagrange multiplier as $p$ goes to infinity. The main difficulty of this part is to prove the convergence of that
sequence in $L^\infty(\Om)$, which is required in order to obtain results supporting the numerical approximation of the multipliers.
We use differential equations methods in this part, therefore the class of considered problems is more 
restrictive than in Theorem \eqref{T:primal-intro}. For those problems we prove the following
\begin{theorem}\label{T:dual-intro}
 Let $u$ be a cluster point of $\{u_p\}_{p\geq p_1}$ in $C(\overline{\Om})$.
 There exists $\lambda\in L^\infty(\Om)$ such that 
\begin{enumerate}
\item The sequence $\{|\nabla u_p|^{p-2}\nabla u_p\}_{p\geq p_1}$ weakly$-*$ converges to $\lambda\nabla u$, up to subsequence. 
\item The primal-dual pair $(u,\lambda)$ satisfy the system
\begin{eqnarray}\label{E:system-intro}
 \Div( W'(|\nabla u|^2)\nabla u)+ \Div( \lambda \nabla u)=-\phi'(u)\ in\ \mathcal{D}'. \label{E:system-intro1} \\
 \lambda(x) \geq 0\ \text{a.e in}\ \Om. \label{E:system-intro2}\\
 \lambda(x)(|\nabla u(x)|-1)=0\ \text{a.e in}\ \Om \label{E:system-intro3}  
\end{eqnarray}
\end{enumerate}
\end{theorem}

For the elastoplastic torsion problem \eqref{P:torsion}, \citet{brezis} proved the uniqueness 
of $\lambda\in L^\infty(\Om)$ verifying \eqref{E:system-intro1}--\eqref{E:system-intro3}. Moreover, using the known
explicit solution for the primal problem on the disk, we obtain an explicit expression for $\lambda$, 
to which the whole sequence $\{ |\nabla u_p|^{p-2}\}_{p\geq p_1}$ must converge. These explicit solutions
make possible to validate numerically our method.

\section{Primal convergence results}\label{sec:primal} 
In this section we provide the proof of Theorem \ref{T:primal-intro}.
The proof is divided into a series of lemmas. For clarity of the exposition we 
put $r=1$, the general case being completely analogous.

\begin{lemma}[Compactness]\label{lemma-compactness}
 we have that:
\begin{enumerate}
\item $\sup_{p\geq s} \frac{1}{p} \|T u_p\|_{p,\Omega}^p<+\infty$, where
$$
\|Tu_p \|_{p,\Omega}^{p}= \int_\Omega T(x,\nabla u_p(x))^{p} dx
$$
\item Let $p_1=\max\{N+1,s\}$. For every $q> 1$, $\{u_p\}_{p\geq p_1}$ is bounded in $W^{1,q}(\Omega;\RR^m)$
\item $\{u_p\}_{p\geq p_1}$ is relatively compact in $C(\overline{\Omega};\RR^m)$.
\item For every uniform cluster point $u_\infty$ of $\{u_p\}_{p\geq
p_1}$, we have that $$u_\infty\in g+ W^{1,\infty}_0(\Omega;\RR^m).$$
\item  If $u_{p_j} \to u_\infty$ in $C(\overline{\Omega};\RR^m)$ then $u_{p_j} \rightharpoonup
u_\infty$ weakly in $W^{1,q}(\Omega;\RR^m)$ for every $q\in[p_1,\infty[$.
\end{enumerate}

\end{lemma}
\begin{proof} From the optimality of $u_p$ 
it follows that
\begin{equation}\label{E:cota from opt}
\alpha+\frac{1}{p}\|T u_p \|_{p,\Omega}^p \leq
J(g)+\frac{1}{p}\|T g \|_{\infty,\Omega}^p
\LN(\Omega),
\end{equation}
where $\alpha= \inf \{J(v)\mid v\in W^{1,s}(\Omega;\RR^m)\}\in\RR$
(recall that $J$ is supposed to be bounded from below). Using \eqref{E:g} we deduce that
$$
\sup_{p\geq s}\frac{1}{p}\|T u_p\|_{p,\Omega}^p<+\infty,
$$
hence
$$
C_1:=\sup_{p\geq s}\|T u_p
\|_{p,\Omega}<+\infty.
$$
In particular,
$$
 \| \nabla u_p\|_{p,\Omega} \leq \alpha_1 C_1.
$$
 On the other hand, the Poincar\'e inequality yields
$$
\Vert u \Vert_{p,\Omega} \leq C(\Omega,p)\left(\|\nabla
u\|_{p,\Omega}+ \| \nabla g \|_{p,\Omega}\right)+\|g\|_{p,\Omega},
$$
for every $u \in g + W^{1,p}_0(\Omega;\RR^m)$ and a suitable
constant $C(\Omega,p)>0$. Combining these estimates, and recalling
\citep{adams} that the constant $C(\Omega,p)$  may be chosen such
that
$$
\sup\limits_{p\in[N+1,\infty[} C(\Omega,p) < +\infty,
$$
 we deduce that there exists a constant $C_2>0$ such that
$$
\forall p\in[
p_1,+\infty[,~\|u_p\|_{1,p,\Omega}=\|u_p\|_{p,\Omega}+\|\nabla
u_p\|_{p,\Omega}\leq C_2,$$ where $p_1=\max\{N+1,s\}$.
 In particular, $\{w_p:=u_p-g \}_{p\geq p_1}$ is bounded in $W^{1,q}_0(\Omega;\RR^m)$ for each $q\geq
p_1$, hence for every $q>1$ by H\"older inequality. Since $p_1>N$, we deduce that $\{w_p\}_{p\geq p_1}$ is
relatively compact in $C(\overline{\Omega};\RR^m)$ by the
Rellich-Kondrachov theorem (since we deal with $W^{1,p_1}_0$ we do
not require any regularity condition on $\partial \Omega$). Thus, we
deduce that $\{u_p\}_{p\geq p_1}$ is relatively compact in
$C(\overline{\Omega};\RR^m)$.

Let $u_\infty$ be a  cluster point of $\{u_p\}_{p\geq p_1}$ in
$C(\overline{\Omega};\RR^m)$. First, we  prove that 
$u_\infty \in W^{1,\infty}(\Omega;\RR^m)$. By Morrey's theorem there exists a
constant $C'(\Omega,p)>0$ such that
\begin{equation*}
|w_p(x)-w_p(y)| \leq C'(\Omega,p)||w_p||_{1,p,\Omega} |x-y|^{1 - N/p}
\end{equation*}
for every $x,y \in \Omega$. In fact, the constant can be chosen in such a way that
$$\sup\limits_{p\in[q,\infty[} C'(\Omega,p) < +\infty$$ for every $q>N$ (see \cite{adams}).
Therefore, we conclude that for a suitable constant $C_3>0$, $|u_p(x)-u_p(y)| \leq C_3 |x-y|^{ 1 -N/p}$, for every
 $x,y\in\Omega$ and $p \in [p_1,\infty[$. We deduce that $$|u_\infty(x)-u_\infty(y)| \leq C_3 |x-y|,$$ then
$u_\infty \in W^{1,\infty}(\Omega;\RR^m)$. Of course, $u_\infty =g$
on $\partial\Omega$.

Next, fix $q \in]1,\infty[$. From our previous analysis it follows
that $\{ u_p \}_{p \in [p_1,\infty[}$ is bounded in
$W^{1,q}(\Omega;\RR^m)$ and therefore   relatively compact for the
weak topology of $W^{1,q}(\Omega;\RR^m)$. Consequently, if $p_j
\to\infty$ is a sequence such that $u_{p_j} \to u_\infty$ uniformly
on $\overline{\Omega}$, then $u_{p_j} \rightharpoonup u_\infty$
weakly in $W^{1,q}(\Omega;\RR^m)$. \end{proof}

\begin{lemma}\label{factibility of the limit} If $u_\infty$ is a cluster point of $\{u_p\mid p\to\infty\}$ in $C(\overline{\Omega};\RR^m)$
then
 $ \|T u_\infty\|_{\infty,\Omega}\leq 1.$
Moreover, $u_\infty$ is an optimal solution to \eqref{e.prob2}, and
we have that
$$\lim\limits_{p\to\infty}J_p(u_p)=\lim\limits_{p\to\infty}J(u_p)=J(u_\infty)=\min
J_\infty.$$
\end{lemma}
\begin{proof}  Let $u_{p_j} \to u_\infty$ in
$C(\overline{\Omega};\RR^m)$ and fix $q\in]1,\infty[$. By Lemma
\ref{lemma-compactness}, $u_{p_j} \rightharpoonup u_\infty$ weakly
in $W^{1,q}(\Omega;\RR^m)$. It follows from the weak lower
semicontinuity in $W^{1,q}(\Omega;\RR^m)$ of $
v\mapsto\|T v \|_{q,\Omega},$ that
$$
\|T u_\infty\|_{q,\Omega}\leq\liminf_{j\to\infty}\|T u_{p_j}\|_{q,\Omega}.
$$
For every $p\in[q,\infty[$, the H\"older inequality yields
$$
\|T u_p\|_{q,\Omega}\leq\|T u_p\|_{p,\Omega} \LN(\Omega)^{\frac{1}{q}-\frac{1}{p}}. $$ Then,
Lemma \ref{lemma-compactness} ensures that
$$
\|T u_p\|_{q,\Omega}\leq
(pC)^{\frac1p}\LN(\Omega)^{\frac{1}{q}-\frac{1}{p}}
$$
for some constant $C>0$. Hence
$$
\|T u_\infty\|_{q,\Omega}\leq\LN(\Omega)^{\frac{1}{q}}
$$
Letting $q\to\infty$, we get the desired inequality.\\

Let $v\in g+W^{1,\infty}(\Omega;\RR^m)$ with $\|T v\|_{\infty,\Omega}\leq 1$. By optimality of $u_p$ we have that
$$
J(u_p)\leq J_p(u_p)\leq J_p(v)=J(v)+\frac{1}{p}\|T v\|_{p,\Omega}^p.
$$
Since $\|T v\|_{\infty,\Omega}\leq 1$, we have that
$$
\limsup_{p\to\infty}J(u_p)\leq \limsup_{p\to\infty} J_p(u_p)\leq
\limsup_{p\to\infty}J_p(v)=J(v).
$$
As $v$ is arbitrary, we obtain that
$$
\limsup_{p\to\infty}J(u_p)\leq \limsup_{p\to\infty} J_p(u_p)\leq
\inf J_\infty.
$$
Now, let $u_{p_j} \to u_\infty$ in $C(\overline{\Omega};\RR^m)$. By
the weak lower semicontinuity of $J$, we have that
$$
J(u_\infty)\leq \liminf_{j\to\infty}J(u_{p_j}),
$$
and due to the previous lemmas, we know that
$J(u_\infty)=J_\infty(u_\infty)$. This proves the optimality of
$u_\infty$ and moreover
$$
\lim_{j\to\infty}J_{p_j}(u_{p_j})=\lim_{j\to\infty}J(u_{p_j})=\min
J_\infty.
$$
Finally, note that, up to a subsequence, the same is valid for an
arbitrary sequence $\{p_k\}_{k\in\NN}$ with $p_k\to \infty$. This
fact together with a compactness argument proves indeed the result.
\end{proof}

\section{Dual convergence results}
In this section we are concerned with the existence and approximation of Lagrange multipliers for the constrained
problem \eqref{e.prob}. The techniques used to this end does not allow the great degree of generality as the primal results 
of Section \ref{sec:primal}.
We shall restrict ourselves to particular cases where the regularity of solutions is known. 
More precisely, we consider the following instances of \eqref{e.prob}
$$
\min \left\{J(v):= \int_\Om [ W(|\nabla v|^2) + \phi(v) ]\ : \ |\nabla v|\leq 1,\ v\in g+H_0^1(\Om) \right\}.
$$
We suppose that $g$ is a real constant, and additionally
\begin{eqnarray}
t\mapsto W(t^2)\ and\ \phi\ \text{are convex and of class}\ C^2(\RR) \label{E:convexity} \\ 
  G(s):= W'(s)+2s W''(s)>0, \quad \mbox{for}\ s>0.  \label{E:ellipticity}
 \end{eqnarray}

Let us consider the penalized problem 
$$
\min \left\{ J(v)+\frac{1}{p} \int |\nabla u|^p : v\in g+H_0^1(\Om) \right\}.
$$
By the convexity assumptions on the functions $W$ and $\phi$, that problem has a unique solution $u_p$ which 
is a weak solution of the Euler-Lagrange equation:
\begin{equation}\label{E:p-equation}
 \Div ( (W'(|\nabla u_p|^2)+|\nabla u_p|^{p-2}) \nabla u_p)=-\phi'(u_p). 
\end{equation}

Let us define:
$$
\Psi(x;\alpha)=\int_0^{|\nabla u_p|^2} G(s) ds + 2 \frac{p-1}{p} |\nabla u_p|^p + \alpha \phi(u_p)
$$

Note that by \eqref{E:ellipticity}, $|\nabla u_p|^p + \alpha \phi(u_p) \leq \Psi(x;\alpha)$. Therefore  if we 
suceed at obtaining uniform bounds for $\Psi$ we can deduce thereof bounds for the sequence $|\nabla u_p|^p$.
Maximum principles of  \citet{Payne-Philippin77,Payne-Philippin79} state that under mild conditions 
the maximum of $\Psi(\cdot,\alpha)$ is attained at a critical point of $u_p$. In such a point $\Psi(x,\alpha)=\alpha \phi(u_p(x))$
and we can conclude using the uniforms bounds on $u_p$ obtained in Section \ref{sec:primal}.
The application of maximum principle techniques require to  work with classical ($C^2(\Om)$) solutions. 
Results of \citet{Uhlenbeck}, \citet{Tolksdorf} and \citet{Lieberman} show that bounded solutions to equations of the type 
\eqref{E:p-equation} are $C^{1,\alpha}(\overline \Om)$-regular, provided that hypothesis \eqref{E:ellipticity} holds. Higher regularity 
can be obtained by a bootstrap argument at points where $\nabla u_p\neq 0$. However, if the function $G$ defined in 
\eqref{E:ellipticity} is degenerate, \emph{i.e.} $G(0)=0$, a further regularization is necessary \citep{Kawohl}.
Following a classic procedure \citep[see eg.][]{Evans-Gangbo,BDBM,SAKAGUCHI,DB} the term $|\nabla u_p|^p$ is regularized 
by $(\eps^2+|\nabla u_p|^2)^{p/2}$ to obtain a sequence of regular functions $u_p^\eps$ converging to $u_p$ pointwise and in $W^{1,p}$ norm  
as $\eps\to 0$. In this way estimations on $u_p$ can be obtained by approximation.

\begin{theorem}
Under hypothesis \eqref{E:convexity}- \eqref{E:ellipticity}, if $\Om$ is convex and $\partial \Om\in C^2$,
then the sequence $\{ |\nabla u_p|^p\}_{p\geq p_1}$ is uniformly bounded in $L^\infty(\Om)$.
 \end{theorem}
\begin{proof}
By \citet[][Corollary 1]{Payne-Philippin79}, the function $\Psi(x;2)$ attains its maximum at a critical point of $u_p$, therefore 
$$
|\nabla u_p|^p + 2 \phi(u_p) \leq \Psi(x;\alpha)\leq \max_{\bar\Om} \Psi(x;2) \leq 2 \max_{\bar\Om} \phi(u_p(x)),
$$
whence 
$$
|\nabla u_p|^p \leq 4 \max_{\bar\Om} \phi(u_p)<+\infty
$$
and conclude by Theorem \ref{T:primal-intro} and the continuity of $\phi$.
\end{proof}

\begin{corollary}\label{C:convergences}
Let $u_\infty\in C^{{0,1}}(\overline \Om)$ be a cluster point of $\{u_p\}_{p\geq p_1}$. Then, passing if neccesary to a further subsequence, 
\begin{enumerate}
\item\label{iii:pointwise-conv} $\nabla u_p(x) \to \nabla u_\infty(x)$ for a.e $x\in\Om$.
 \item\label{i:convergence-grad} $\nabla u_p \weakstar \nabla u_\infty$ in the weak$-*$ topology of $L^\infty(\Om)$.
 \item\label{ii:convergence-lambda} There exists $A\in L^\infty(\Om)^n$ such that the sequence $\{|\nabla u_p|^{p-2}\nabla u_p\}_{p\geq p_1}$
converges to $A$ in the weak$-*$ topology.
\end{enumerate}
 \end{corollary}
\begin{proof}
 Assertion \ref{iii:pointwise-conv} is obtained from  \citet{Boccardo1992} using hypothesis \eqref{E:convexity}. 
 Points \ref{i:convergence-grad} and \ref{ii:convergence-lambda} are consequences of the Banach–Alaoglu Theorem. 
\end{proof}
\vspace{\baselineskip}
 
 We are now in position to state our existence and approximability result for both primal and dual solutions of 
 \eqref{e.prob}
 
\begin{theorem}\label{T:dual}
 Let $u$ be a cluster point of $\{u_p\}_{p\geq p_1}$ in $C(\overline{\Om})$ achieving the convergences of Corollary \ref{C:convergences}. 
 There exists $\lambda\in L^\infty(\Om)$ such that 
\begin{enumerate}
\item The sequence $\{|\nabla u_p|^{p-2}\nabla u_p\}_{p\geq p_1}$ weakly$-*$ converges to $\lambda\nabla u$. 
\item The primal-dual pair $(u,\lambda)$ satisfy the system
\begin{eqnarray}
 \Div( W'(|\nabla u|^2)\nabla u)+ \Div ( \lambda \nabla u)=-\phi'(u)\ in\ \mathcal{D}'.\label{complementarity-1} \\
 \lambda(x) \geq 0\ \text{a.e in}\ \Om.\label{complementarity-2} \\
 \lambda(x)(|\nabla u(x)|-1)=0\ \text{a.e in}\ \Om  \label{complementarity-3}
\end{eqnarray}

\end{enumerate}
\end{theorem}
\begin{proof}
The first step of the proof consists in showing that the limit field $A$ in Corollary \ref{C:convergences}~\ref{ii:convergence-lambda}
verifies
\begin{equation}\label{E:representation-of-A}
 |A|=A\cdot \nabla u\quad \text{a.e in}\ \Om.
\end{equation}
Using $u-g$ as test function in \eqref{E:p-equation} we have
\begin{equation}
 \int_\Om |\nabla u_p|^{p-2}\nabla u_p \nabla u=-\int_\Om W'(|\nabla u_p|^{2})\nabla u_p \nabla u+\phi'(u_p)(u-g)
\end{equation}
Then by Corollary \ref{C:convergences}
\begin{equation}
 \int_\Om A \nabla u=-\int_\Om W'(|\nabla u|^{2})|\nabla u|^{2}+\phi'(u)(u-g).
\end{equation}

The same procedure using $u_p-g$ as test shows that
\begin{equation}
 \int_\Om |\nabla u_p|^{p} \longrightarrow-\int_\Om W'(|\nabla u|^{2})|\nabla u|^{2}+\phi'(u)(u-g).
\end{equation}
and therefore
\begin{equation}
 \int_\Om |\nabla u_p|^{p} \longrightarrow \int_\Om A \nabla u
\end{equation}
whence 
\begin{equation}
 \int_\Om |A|\leq \int_\Om A\cdot \nabla u
\end{equation}
and \eqref{E:representation-of-A} follows using $|\nabla u|\leq 1$ a.e (Theorem \ref{T:primal-intro}).
The existence of $\lambda\in L^\infty(\Om)$ satisfying \eqref{complementarity-2} \& \eqref{complementarity-3}
follows from \eqref{E:representation-of-A}. Taking the limit in \eqref{E:p-equation} using Theorem \ref{T:primal-intro},
Corollary \ref{C:convergences} and the representation \eqref{E:representation-of-A} gives \eqref{complementarity-3}.
\end{proof}

\section{Numerical experiments}
We solved numerically the elastoplastic torsion problem in a variety of domains, that 
permitted to gain some insight on the method.
The problem 
\begin{equation}\label{esp-E:problema-general}
\min \left \{ \frac{1}{2}\int |\gr u|^2-\int hu \left | \begin{array}{c} |\gr u|\leq 1\ a.e\ in\ \Om\\ u=0\ \text{on}\
\partial\Om \end{array} \right\} \right.
\end{equation}
is approximated by the sequence of unconstrained problems
\begin{equation}\label{esp-E:problema-penalizado}
\min \left \{ \frac{1}{2}\int |\gr u_p|^2+\frac{1}{p}\int |\gr u_p|^p -\int hu_p\ \Big|\  u_p=0\ \text{on}\
\partial\Om  \right\}
\end{equation}
which possess an unique and regular solution. Besides, results of \citet{brezis-stamp} ensure that solutions $u_\infty$
of \eqref{esp-E:problema-general} are of class $C^1(\overline\Om)$ for regular domains.
For Problem \eqref{esp-E:problema-penalizado} we solve the Euler equation
\begin{equation*}
\int \pto{\gr u_p}{\gr v}+\int |\gr u_p|^{p-2}\pto{\gr u_p}{\gr v} -\int hv=0\quad \forall v\in\mathcal V,
\end{equation*}
where $\mathcal V$ stands for the space of continuous functions whose restriction to any element of a regular mesh
of $\Om$ is polynomial of degree $2$.  Since we are dealing with a nonlinear problem, we cannot apply the finite 
elements method directly; the use of an iterative procedure is necessary. However, for large $p$ 
the  convergence and stability of such an iterative procedure is a delicate issue.
\citet{ruo} proposed to use the term $|\nabla u_p|^{p-2}$ as a preconditioner 
in a gradient-type method with good results (cf. Algorithm \ref{A:primal-dual}).
Incidentally, the term used as a preconditioner by \citet{ruo} coincides with the approximating multiplier, 
and therefore, in the light of Theorem \ref{T:dual}, their algorithm can be viewed as a primal-dual algorithm 
with a multiplier computed explicitly from the primal solution, instead of maximizing a saddle-point function. 
For the tests presented here, we implemented Algorithm \ref{A:primal-dual} in C++ using the 
deal.II finite elements library \citep{dealII}. 

\begin{algorithm}[t]
\begin{tabular}{l@{\extracolsep{.5cm}}p{11cm}}
& Given $p>2$ and an initial point $u_0$, choose $c_1,\eps$. Set $n := 0$ and iterate: \\
\midrule
1.& Compute the multiplier $\lambda_n=|\nabla u_n|^{p-2}$.\\
2.& Find the primal descent direction  $w_n$, by solving
\begin{equation*}
\int_\Om (1+\lambda_n)\gr w_n \gr v dx=-\int_\Om \lambda_n \gr u_n \gr v dx + \int_\Om
fv dx \ \forall v
\end{equation*}\\
3.& Perform a line-search with sufficient decrease condition, \emph{i.e}, \\
   &find  $\alpha_n>0$ satisfying $ J(u_n+\alpha_n w_n)\leq J(u_n)+ c_1 \alpha_n J'(u_n)[w_n]$ \\
4. & Set $u_{n+1}=u_n+\alpha_n w_n$.\\
5. & If $\|J'(u_{n+1})\|\leq \eps$, stop. Otherwise update $n = n + 1$ and go to step 1.\\
\end{tabular}
\caption{The primal-dual algorithm }\label{A:primal-dual}
\end{algorithm}

Denote by {\em D} the unit disk of $\RR^2$, i.e $D=\{ x\in\RR^2 \mid x_1^2+x_2^2< 1\}$.
When $\Om=D$ and $h$ is constant,  \eqref{esp-E:problema-general} has an explicit solution \citep{GLT}. For $h\equiv4$ the solution 
is given by:
\begin{equation}\label{E:explicit-solution}
u(x)=\begin{cases}
 1-r & if\ 1/2 \leq r \leq 1\\
-r^2+3/4 & if\ 0\leq r\leq 1/2 
     \end{cases}
 \end{equation}
where $r=\sqrt{x^2+y^2}$. Since $\Om$ is convex, in this case the multiplier $\lambda$ is continuous and it is 
also given by an explicit formula \citep{brezis},
\begin{equation}\label{E:explicit-multiplier}
\lambda(x)=\begin{cases}
2 r-1  & if\ 1/2 \leq r \leq 1\\
0 & if\ 0\leq r\leq 1/2 
     \end{cases}
 \end{equation}
The norm of the gradient of the computed solution and the multiplier are plot in Figure \ref{F:circle}.  
In Table \ref{Tb:disk} we show the error with respect to the explicit solutions for different values of $p$.
It is shown that for a working precision, a parameter $p$ in the order of few hundreds is enough, preserving 
in this way the numerical stability of the algorithm.

Solving the problem in different domains gives some intuition about the extensibility of Theorem \ref{T:dual} to 
more general situations. In Figures \ref{F:rectangle} and \ref{F:Glowinski} we show the solutions of Problem 
\ref{esp-E:problema-general} in a rectangle and a domain with an interior corner, respectively. We also plot the approximate
multipliers. It is seen that in the rectangle, a convex domain with piecewise smooth border, we are still able 
to compute satisfactorily both the solution and the multiplier. On the contrary, in the piecewise smooth nonconvex
domain, even if the are able to compute the solution with a good accuracy, it is not enough to have the 
multiplier bounded. The difficulty relies on the concentration effect occuring near the interior corners. 
However, the plot with a truncated scale shows that far from the concentrations we are computing the 
right multiplier, suggesting that our method combined with some truncation mechanism \citep[see eg.][Section 4]{Li} should be able to 
cope with a more general class of problems.

\begin{table}
\caption{Error of $u_p$ and $\lambda_p$ with respect to the respective primal and dual analytical solutions of 
the limit problem given in \eqref{E:explicit-solution} and \eqref{E:explicit-multiplier}  in various norms.}
\label{Tb:disk}
\begin{center}
\begin{tabular}{ll@{\hspace{.25cm}}ll@{\hspace{.25cm}}l@{\hspace{.25cm}}l@{\hspace{.15cm}}l@{\hspace{.25cm}}l}
\toprule
    &    \multicolumn{2}{c}{Mesh info}  &\multicolumn{3}{c}{Primal error} & \multicolumn{2}{c}{Dual error}\\
$p$ & \# cells & \# dofs & $L^2$-norm & $H^1$-norm & $W_0^{1,\infty}$-norm & $L^1$-norm & $L^\infty$-norm \\ \midrule
10  & 65708 & 280049 & 4.585e-02 & 1.355e-01 & 1.756e-01 & 2.645e-01 & 2.133e-01 \\
50  & 65348 & 273345 & 9.876e-03 & 3.255e-02 & 5.680e-02 & 5.115e-02 & 6.058e-02 \\
100 & 123917 & 517501 & 4.989e-03 & 1.705e-02 & 3.366e-02 & 2.555e-02 & 3.530e-02 \\
300 & 442940 & 1883001 & 1.674e-03 & 5.956e-03 & 1.416e-02 & 8.716e-03 & 2.963e-02 \\
500 & 857396 & 3698513 & 1.006e-03 & 3.624e-03 & 9.358e-03 & 5.267e-03 & 2.705e-02 \\
 \bottomrule
\end{tabular}
\end{center}
\end{table}

\begin{figure}
\includegraphics[width=.5\linewidth]{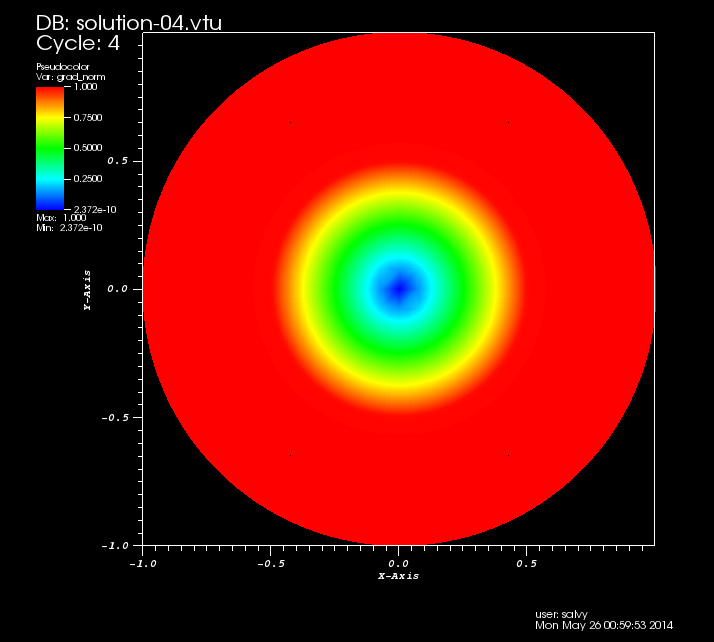}
\includegraphics[width=.5\linewidth]{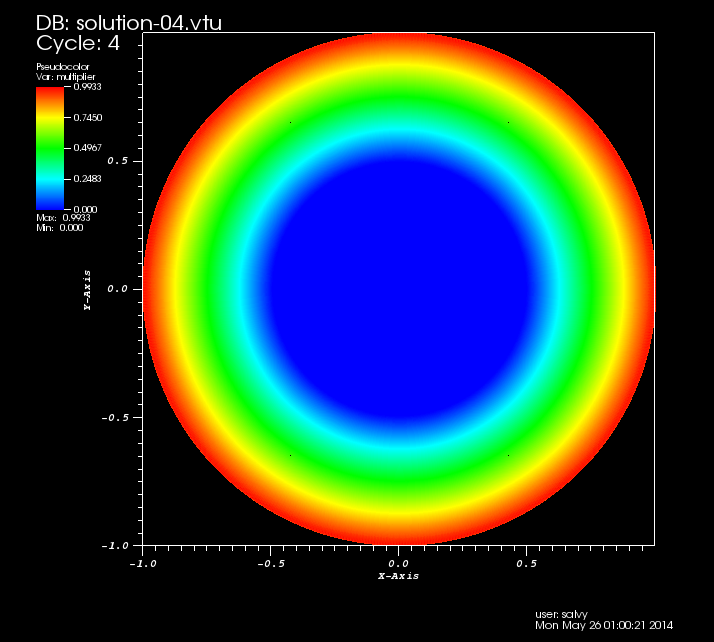}
  \caption{Plot of the norm of the gradient $|\nabla u_p|$ and the multiplier $\lambda_p=|\nabla u_p|^{p-2}$ for $p=500$
  on a circle.}
\label{F:circle}
\end{figure}
\begin{figure}
\includegraphics[width=.5\linewidth]{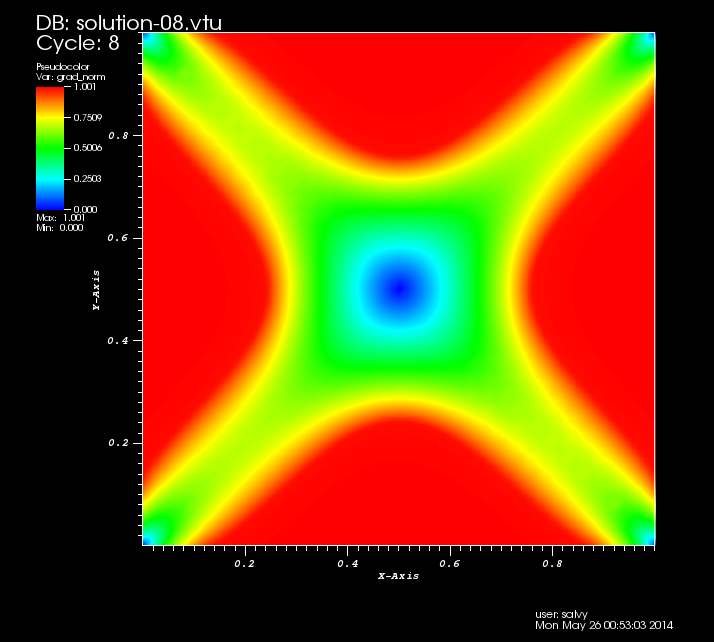}
\includegraphics[width=.5\linewidth]{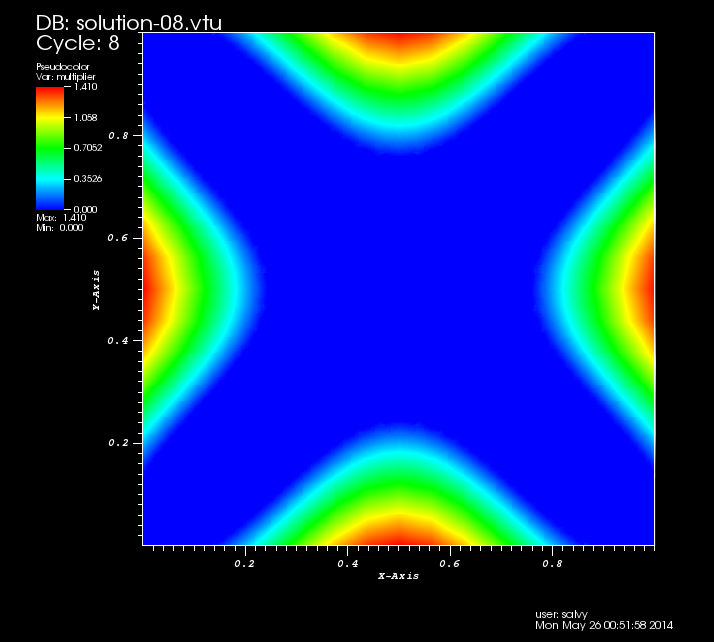}
  \caption{Plot of the norm of the gradient $|\nabla u_p|$ and the multiplier $\lambda_p=|\nabla u_p|^{p-2}$ for $p=300$
  on a rectangle.}\label{F:rectangle}
\end{figure}
\begin{figure}
\includegraphics[width=.5\linewidth]{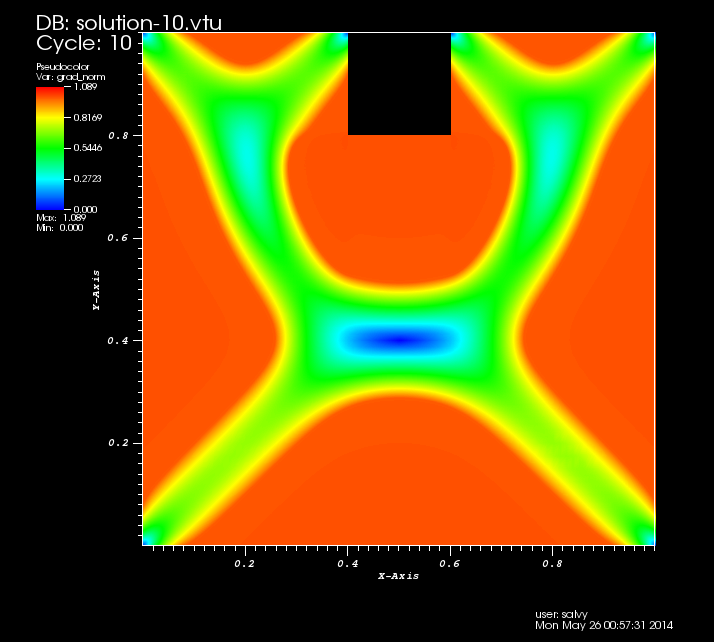}
\includegraphics[width=.5\linewidth]{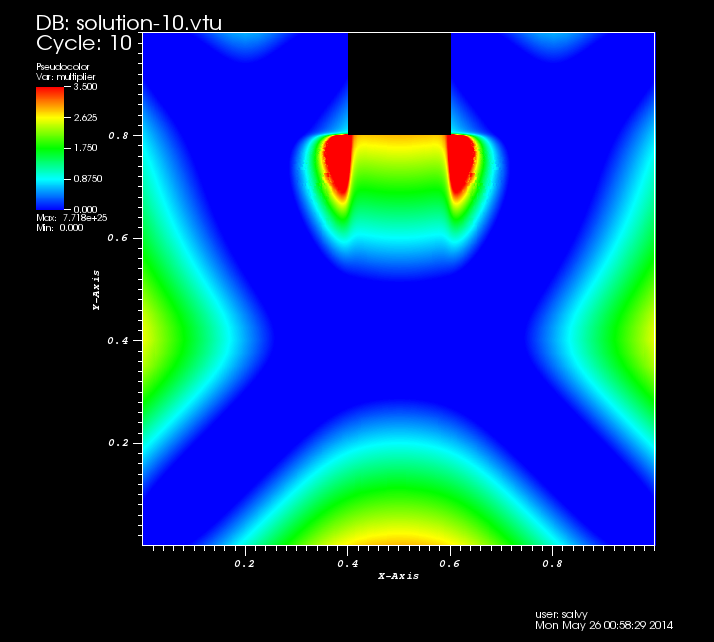}
  \caption{Plot of the norm of the gradient $|\nabla u_p|$ and the multiplier $\lambda_p=|\nabla u_p|^{p-2}$ for $p=700$ 
  on a domain with an interior corner. The scale in the plot of the multiplier is truncated.}
\label{F:Glowinski}
\end{figure}
\bibliographystyle{svjour}
\bibliography{biblio-mult}

\end{document}